\documentclass[a4paper,oneside]{amsart}
\usepackage{amssymb}
\usepackage{amsthm}
\usepackage{amsmath,amscd}
\usepackage[mathscr]{euscript}
\usepackage[all]{xy}
\usepackage[utf8]{inputenc}
\usepackage{lmodern}
\usepackage[T1]{fontenc}
\usepackage[textwidth=14cm,hcentering]{geometry}
\usepackage[colorlinks=true,linkcolor=red,citecolor=blue]{hyperref}
\usepackage{mathtools}
\usepackage{tikz}
\usetikzlibrary{arrows,positioning}

\newtheorem{theo}{Theorem}[section]
\newtheorem{lemm}[theo]{Lemma}
\newtheorem{prop}[theo]{Proposition}
\newtheorem{coro}[theo]{Corollary}
\newtheorem{fact}[theo]{Fact}

\theoremstyle{definition}
\newtheorem{defi}[theo]{Definition}
\newtheorem{cons}[theo]{Construction}
\newtheorem{assu}[theo]{Assumption}

\newtheorem{example}[theo]{Example}

\newtheorem{rem}[theo]{Remark}

\newtheorem*{theo*}{Theorem}

\numberwithin{equation}{section}

\newcommand{\cat}{\mathbf}
\newcommand{\on}{\operatorname}

\newcommand{\Hom}{\mathrm{Hom}}

\newcommand{\C}{\mathbf{C}}
\newcommand{\Q}{\mathbf{Q}}
\newcommand{\Z}{\mathbf{Z}}
\newcommand{\A}{\mathbf{A}}

\renewcommand{\H}{\mathrm{H}}

\newcommand{\D}{\mathbf{D}}
\newcommand{\DA}{\mathbf{DA}}
\newcommand{\DAT}{\mathbf{DAT}}
\newcommand{\DM}{\mathbf{DM}}
\newcommand{\DMT}{\mathbf{DMT}}

\newcommand{\PConf}{\mathrm{PConf}}
\newcommand{\Conf}{\mathrm{Conf}}

\newcommand{\R}{\mathbf{R}}
\newcommand{\N}{\mathbf{N}}
\newcommand{\rmod}[1]{\ensuremath{#1\text{-}\mathrm{Mod}}}
\newcommand{\htop}{\mathrm{hTop}}

\title{Motivic homological stability of configuration spaces}
\author{Geoffroy Horel}
\author{Martin Palmer}

\begin{document}

\address{Université Sorbonne Paris Nord, Laboratoire Analyse, Géométrie et Applications, CNRS (UMR 7539), 93430, Villetaneuse, France.}
\email{horel@math.univ-paris13.fr}
\address{Mathematical Institute of the Romanian Academy, 21 Calea Griviței, 010702 București, România}
\email{mpanghel@imar.ro}

\begin{abstract}
We prove that some of the classical homological stability results for configuration spaces of points in manifolds can be lifted to motivic cohomology.
\end{abstract}

\maketitle

\section{Introduction}

A classical result of McDuff and Segal states that the unordered configuration spaces of a connected, open manifold $M$ are homologically stable. More precisely, let $M$ be a connected manifold homeomorphic to the interior of a manifold with non-empty boundary and write
\[
C_n(M) = \{ (x_1,\ldots,x_n) \in M^n \mid x_i \neq x_j \text{ for } i \neq j \} / \Sigma_n
\]
for the \emph{unordered configuration space} of $n$ points in $M$. The theorem of McDuff and Segal (reproven more recently by Randal-Williams) is the following.

\begin{theo}[\cite{segalconfiguration,mcduffconfiguration,segaltopology,Randal-Williams2013}]
\label{t:classical}
There are stabilization maps of the form $C_n(M) \to C_{n+1}(M)$ such that the induced maps on integral homology
\[
 \H_i(C_n(M);\Z) \longrightarrow \H_i(C_{n+1}(M);\Z) 
\]
are isomorphisms when $n \geq 2i$. 
\end{theo}

This is part of a more general phenomenon of homological stability that holds for many other families of spaces or groups, including general linear groups \cite{Kallen1980}, mapping class groups of surfaces \cite{harerstability}, automorphism groups of free groups \cite{Hatcher1995, HatcherVogtmann1998} and moduli spaces of high-dimensional manifolds \cite{GalatiusRandal-Williams2017, GalatiusRandal-Williams2018}.

The goal of this paper is to lift Theorem \ref{t:classical} to a version for {\'e}tale and motivic cohomology. If $X$ is a smooth scheme over a number field $K$, we may consider its associated unordered configuration scheme $\mathrm{Conf}_n(X)$, as well as a stacky version $\widetilde{\Conf}_n(X)$ (described precisely in \S\ref{quotient}). When $X$ is not complete (and under a minor additional assumption), we construct, in \S\ref{stabilization}, stabilization maps
\begin{equation}
\label{eq:stab}
M(\widetilde{\Conf}_n(X)) \longrightarrow M(\widetilde{\Conf}_{n+1}(X))
\end{equation}
in the category of motives. For any choice of embedding of $K$ into $\mathbb{C}$, their Betti realizations agree with the classical stabilization maps for the unordered configuration spaces of the complex manifold $M = X_{an}$ consisting of the complex points of $X$.

We assume that $X$ may be written as $X \cong Y-D$, where $Y$ is a smooth scheme over $K$ and $D \subset Y$ is a smooth closed subscheme that has a $K$-point. Our main result is then the following.

\begin{theo}[Theorems \ref{theorem-open-schemes} and \ref{main-theorem}]
\label{t:main}
Let $X$ be as above and assume that its {\'e}tale motive $M_{et}(X)$ is mixed Tate. Assume also that $Y$ is geometrically connected. Then the maps of {\'e}tale motivic cohomology groups
\[
\H^{p,q}_{et}(\widetilde{\Conf}_{n+1}(X);\Lambda) \longrightarrow \H^{p,q}_{et}(\widetilde{\Conf}_n(X);\Lambda)
\]
induced by \eqref{eq:stab} are isomorphisms for $p \leq n/2$ and under mild conditions on the coefficient ring $\Lambda$ (see Theorem \ref{theo: Betti detects connectivity} for the precise conditions). If $X$ is a quasi-projective variety, there are isomorphisms of {\'e}tale motivic cohomology groups
\[
\H^{p,q}_{et}(\widetilde{\Conf}_n(X),\Lambda)\cong \H^{p,q}_{et}(\Conf_{n}(X),\Lambda)
\]
for any $\Lambda$, so in this case the result may be interpreted as a statement about the non-stacky configuration schemes $\mathrm{Conf}_n(X)$. In the case when $X = \A^d$ is affine space, the maps of motivic cohomology groups
\[
\H^{p,q}(\widetilde{\Conf}_{n+1}(\A^d);\Lambda) \longrightarrow \H^{p,q}(\widetilde{\Conf}_n(\A^d);\Lambda)
\]
induced by \eqref{eq:stab} are isomorphisms for $p \leq n/2 - 1$ and any coefficient ring $\Lambda$.
\end{theo}

The statement for {\'e}tale motivic cohomology is proven in \S\ref{etale} as a direct consequence of the topological case (Theorem \ref{t:classical}), a detection result for Betti realization (Theorem \ref{theo: Betti detects connectivity}) and the existence of the stabilization maps at the {\'e}tale motivic level. The proof of the statement for motivic cohomology in the case $X = \A^d$ is proven in \S\ref{motivic}, and is more indirect, since the analogous detection result does not hold in this case. Instead, we use a detection result for the associated graded of the weight filtration, and the key topological input (Lemma \ref{lem:twisted-homstab}) is a certain twisted homological stability result for the symmetric groups. 

Let us emphasize that the first part of the theorem applies in particular to $X=\A^d$ and gives us an étale motivic cohomological stability result. This is different from the second part of the theorem which is a cohomological stability result for motivic cohomology (as opposed to étale motivic cohomology). The price to pay for this finer result is that we have to work with $\widetilde{\Conf}_n$ instead of $\Conf_n$.

\begin{rem}
When $X$ is the affine line $\A^1$, the first statement of Theorem \ref{t:main} (for {\'e}tale motivic cohomology) has previously been proven in \cite{horelmotivic},\footnote{In fact, \cite{horelmotivic} also claimed to prove the second statement of Theorem \ref{t:main} (for motivic cohomology) for the affine line $X = \A^1$, but the proof contained an error; see the erratum \cite{HorelErratum}.} except that the isomorphisms (in a stable range) of \cite{horelmotivic} are induced not by motivic lifts of \emph{stabilization maps}, but rather by motivic lifts of \emph{scanning maps} $C_n(\C) \to \Omega^2_n S^2$ and of maps between the different path-components of $\Omega^2 S^2$. The paper \cite{horelmotivic} also uses a different model (denoted by $C_n$ in \cite[\S 5]{horelmotivic}) for the configuration scheme $\mathrm{Conf}_n(\A^1)$.
\end{rem}

Throughout this paper, we denote by $K$ a number field equipped with an embedding $K\to \mathbb{C}$ and we denote by $S$ the spectrum of $K$. For a smooth scheme $X$ over $S$, we denote by $X_{an}$ the set $X(\mathbb{C})$ with its complex manifold structure. 

\bigskip
\noindent\textbf{Acknowledgements.}
We are grateful to the anonymous referee for their helpful comments and careful reading of an earlier version of this article. The authors were both partially supported by a grant of the Romanian Ministry of Education and Research, CNCS - UEFISCDI, project number PN-III-P4-ID-PCE-2020-2798, within PNCDI III and by a grant of the Agence Nationale pour la Recherche, project number ANR-18-CE40-0017 PerGAMo.

\section{Quotient and stacky quotient}
\label{quotient}

Let $X$ be a smooth scheme over $S$ and $G$ be a finite group acting on the right on $X$. We denote by $X/G$ the quotient of $X$ by $G$ in the category of smooth schemes whenever it exists. A sufficient condition for this to be true is for $X$ to be a quasi-projective variety (by \cite[\href{https://stacks.math.columbia.edu/tag/07S7}{Tag 07S7}]{stacksproject}) and the action of $G$ to be free.

 We denote by $[X/G]$ the ``stacky quotient'' of $X$ by $G$. This is a simplicial object in the category of smooth $S$-schemes given by 
\[[n]\mapsto X\times G^n.\]
This is the nerve of the translation groupoid of the action of $G$ on $X$.

By mapping a fixed smooth scheme into a simplicial scheme, we can turn a simplicial scheme into a simplicial presheaf on the category of smooth schemes. We will use the same notation to denote a simplicial scheme and the associated simplicial presheaf. Observe that there is a canonical map of simplicial schemes $[X/G]\to X/G$ (where we view $X/G$ as a constant simplicial scheme).

\begin{prop}\label{prop : quotient vs stacky quotient}
Assume that  $X$ is a quasi-projective variety over $S$ equipped with a free action of a finite group $G$. Then the canonical map $[X/G]\to X/G$ is an étale weak equivalence of simplicial presheaves. 
\end{prop}

\begin{proof}
First observe that under these conditions, the quotient exists in the category of schemes and the quotient map $X\to X/G$ is an \'etale map by \cite[\href{https://stacks.math.columbia.edu/tag/07S7}{Tag 07S7}]{stacksproject}. Let $F$ be a fibrant object in simplicial presheaves with the Jardine model structure (\cite{jardinesimplical}). We denote by $f$ the map induced by precomposition with the canonical map
\[f:\mathrm{Map}(X/G,F)\to \mathrm{Map}([X/G],F).\]
We need to prove that $f$ is a weak equivalence of simplicial sets. The domain of $f$ is simply $F(X/G)$ while the codomain of $f$ is the homotopy limit of the cosimplicial diagram
\[[n]\mapsto F(X\times G^n)\] 

Now we observe that there is an isomorphism 
\[X\times_{(X/G)}X\times_{(X/G)}\ldots\times_{(X/G)}X\cong X\times G^n\]
where the fiber product on the left has $n+1$ terms and this isomorphism is compatible with the cosimplicial structures on both sides (this isomorphism appears in the proof of \cite[\href{https://stacks.math.columbia.edu/tag/07S7}{Tag 07S7}]{stacksproject}). So the map $f$ can be identified with the map
\[F(X/G)\to\mathrm{holim}_{\Delta}([n]\mapsto F(X^{\times_{(X/G)} (n+1)})).\]
The latter map is a weak equivalence since $F$ is fibrant and $X\to X/G$ is an étale cover.
\end{proof}

For $X$ a smooth scheme over $S$, we denote by $\PConf_n(X)$ the complement in $X^n$ of all the diagonals:
\[\PConf_n(X)=X^n-\cup_{i,j}\Delta_{i,j}\]
where $\Delta_{i,j}=\{(x_1,\ldots,x_n)\in X^n, x_i=x_j\}$. The $\Sigma_n$-action on $X^n$ restricts to a $\Sigma_n$-action on $\PConf_n(X)$. If $X$ is a quasi-projective smooth variety, we denote by $\Conf_n(X)$ the quotient of $\PConf_n(X)$ by the action of $\Sigma_n$. Finally, we denote by $\widetilde{\Conf}_n(X)$ the stacky quotient $[\PConf_n(X)/\Sigma_n]$. Observe that in this case the assumptions of the previous proposition are satisfied.

\begin{rem}
The letter $\mathrm{P}$ in the notation $\PConf$ stands for the word ``pure''; the reason for this notation is that, if $X$ is a smooth complex curve, the topological fundamental group of $\PConf_n(X)$ is the pure braid group of the associated Riemann surface. 
\end{rem}

\section{Motives}

\subsection{Generalities}

We shall recall briefly here the construction of the categories of motives $\DM(S,\Lambda)$ and $\DA(S,\Lambda)$. We refer the reader to \cite{mazzalecture} or \cite{cisinskitriangulated} for details about these categories. We start from the category of complexes of presheaves of $\Lambda$-modules on the site of smooth schemes over $S$ and we force descent for étale hypercovers, contractibility of the affine line and invertibility of the Tate motive for the tensor product. The resulting triangulated category is denoted $\DA(S,\Lambda)$. We define $\DM(S,\Lambda)$ in a similar fashion except that we start from the category of complexes of presheaves with transfers and we use the Nisnevich topology instead of the étale topology. Note that we could also work with presheaves with transfers and the \'etale topology and construct the category $\DM_{et}(S,\Lambda)$; however, this category is equivalent to $\DA(S,\Lambda)$ by \cite[Corollaire B.14]{ayoubalgebreI}. There is a left adjoint functor
\[\DM(S,\Lambda)\to\DA(S,\Lambda)\]
simply given by the \'etale sheafification functor $\DM(S,\Lambda)\to \DM_{et}(S,\Lambda)$ followed by the equivalence $\DM_{et}(S,\Lambda)\simeq \DA(S,\Lambda)$. This left adjoint is an equivalence of categories when $\Lambda$ is a $\mathbb{Q}$-algebra. The category $\DM(S,\Lambda)$ contains a collection of objects indexed by $\mathbb{Z}$ called the Tate twists and denoted by $\Lambda(n),n\in\mathbb{Z}$. There are analogously defined motives in $\DA(S,\Lambda)$. There is a symmetric monoidal category structure on both $\DM(S,\Lambda)$ and $\DA(S,\Lambda)$ such that the formula
\[\Lambda(i)\otimes\Lambda(j)\cong\Lambda(i+j)\]
holds for all integers $i$ and $j$.

A smooth scheme $X$ over $S$ yields an object in $\DM(S,\Lambda)$ and in $\DA(S,\Lambda)$ denoted $M(X)$ and $M_{et}(X)$ respectively. The étale motivic cohomology with coefficients in $\Lambda$ of a smooth scheme $X$ over $S$ is the bigraded collection of $\Lambda$-modules:
\[\H^{p,q}_{et}(X,\Lambda):=\Hom_{\DA(S,\Lambda)}(M_{et}(X),\Lambda(q)[p])\]
where $[p]$ denotes the shift by $p$ in a triangulated category. Likewise the motivic cohomology is given by 
\[\H^{p,q}(X,\Lambda):=\Hom_{\DM(S,\Lambda)}(M(X),\Lambda(q)[p])\]
We can also define the motivic cohomology of $[X/G]$ when $G$ is a finite group acting on a smooth scheme $X$. For this we first define $M([X/G])$ as the homotopy colimit of the simplicial object of $\DM(S,\Lambda)$ given by
\[[n]\mapsto M(X\times G^n)\]
and then we set
\[\H^{p,q}([X/G],\Lambda):=\Hom_{\DM(S,\Lambda)}(M([X/G]),\Lambda(q)[p]).\]

\begin{rem}
The careful reader will have noticed that there is an abuse of terminology in the definition above as there is no definition of the homotopy colimit of a simplicial object in a triangulated category. What we really mean is that $\DM(S,\Lambda)$ is the homotopy category of a model category (or an $\infty$-category) and that the diagram $[n]\mapsto M(X\times G^n)$ lifts to the level of model categories. We can thus take the homotopy colimit in the model category and then consider the result as an object in the homotopy category. We will allow ourselves to make this abuse in a few other places in the paper.
\end{rem}

\begin{lemm}
\label{lem:motive-of-stacky-quotient}
The object $M([X/G])$ is weakly equivalent to the homotopy orbits of the $G$-action on $M(X)$.
\end{lemm}

\begin{proof}
Recall that any category $\cat{M}$ with colimits is tensored over the category of sets via the formula
\[S\square X:= \bigsqcup_SX\]
where $S$ is a set, $X$ is an object of $\cat{M}$ and $\square$ is our notation for the tensoring.

Now assume that $\cat{M}$ is a model category and $X$ has a left action of a group $G$. Then, it is standard that the homotopy orbits of the $G$-action on $X$ are given by the homotopy colimit of the usual simplicial bar construction:
\[[n]\mapsto G^n\square X.\]

Now, in the particular case of $\cat{M}=\DM(S,\Lambda)$, we see that this simplicial object is isomorphic to the one that appeared in our definition of the motive of the stacky quotient $[X/G]$.
\end{proof}

We could define in a similar fashion the étale motive of $[X/G]$. By Proposition \ref{prop : quotient vs stacky quotient}, we have an isomorphism
\[M_{et}([X/G])\cong M_{et}(X/G)\]
whenever the assumptions of this proposition are satisfied. Since étale motivic cohomology coincides with motivic cohomology when the ring of coefficients is a $\mathbb{Q}$-algebra, we can also deduce the following theorem.

\begin{theo}
Let $\Lambda$ be a $\mathbb{Q}$-algebra. Under the assumptions of Proposition \ref{prop : quotient vs stacky quotient}, the canonical map $[X/G]\to X/G$ induces an isomorphism
\[\H^{p,q}(X/G,\Lambda)\to \H^{p,q}([X/G],\Lambda).\]
\end{theo}

Another important feature of these categories that we now recall is the so-called purity theorem.

\begin{theo}
Let $X$ be a smooth scheme and $D$ be a smooth closed subscheme of $X$ of codimension $c$, then there exists a cofiber sequence in $\DM(S,\Lambda)$.
\[M(X-D)\xrightarrow{M(i)} M(X)\to M(D)(c)[2c]\]
where $i$ denotes the open inclusion $X-D\to X$. There is a similar sequence in $\DA(S,\Lambda)$.
\end{theo}

\subsection{Betti realization}

The functor $X\mapsto X_{an}$ from smooth $S$-schemes to topological spaces induces a functor called Betti realization
\[B^*:\DA(S,\Lambda)\to \D(\Lambda)\]
and similarly a functor
\[B^*:\DM(S,\Lambda)\to \D(\Lambda)\]
We refer the reader to \cite{ayoubalgebreI} and \cite{ayoubalgebreII} for details about these two functors. The only things we will need to know about these functors is that they are symmetric monoidal left adjoints (in fact they come from left Quillen functors) and:
\begin{fact}
\label{fact-Betti-realisation}
The composite
\[B^*\circ M:\mathrm{Sm}_S\to\D(\Lambda)\]
is naturally isomorphic to the functor $X\mapsto C_*(X_{an},\Lambda)$, where $\mathrm{Sm}_S$ denotes the category of smooth schemes over $S$.
\end{fact}

\subsection{Mixed Tate motives}

We denote by $\DAT(S,\Lambda)$ the smallest triagulated subcategory of $\DA(S,\Lambda)$ containing all the Tate twists $\Lambda(n)$ and closed under arbitrary coproducts and retracts. We define $\DMT(S,\Lambda)$ analogously. An object of $\DAT(S,\Lambda)$ or $\DMT(S,\Lambda)$ will be called a mixed Tate motive. Observe that the tensor product on $\DM(S,\Lambda)$ and $\DA(S,\Lambda)$ induces a symmetric monoidal structure on $\DMT(S,\Lambda)$ and $\DAT(S,\Lambda)$ respectively. We denote by $\DAT(S,\Lambda)_{gm}$ (resp. $\DMT(S,\Lambda)_{gm}$) the smallest thick subcategory of $\DAT(S,\Lambda)$ (resp. $\DMT(S,\Lambda)$)  that contains all the Tate twists $\Lambda(i), i\in \mathbb{Z}$. The objects of this category are called geometric mixed Tate motives.

\begin{prop}
Let $X$ be a smooth scheme over $S$ that is such that $M(X)$ is a geometric mixed Tate motive; then for each $n$, the motive $M(\PConf_n(X))$ is also a geometric mixed Tate motive. Similarly, if $M_{et}(X)$ is a geometric mixed Tate \'etale motive, then so is $M_{et}(\PConf_n(X))$.
\end{prop}

\begin{proof}
The proof in $\DMT$ and $\DAT$ are similar. Let $P(n)$ be the set of subsets of $\{1,\ldots,n\}$ of cardinality $2$. For $P$ a subset of $P(n)$, we denote by $\PConf_{n,P}(X)$ the complement in $X^n$ of the diagonals indexed by $P$. We have $\PConf_{n,\varnothing}(X)=X^n$ and $\PConf_{n,P(n)}(X)=\PConf_n(X)$. We shall prove more generally that $M(\PConf_{n,P}(X))$ is mixed Tate for all $P$. We do this by induction on the pair $(n,|P|)$ with the lexicographic ordering. The result is obvious for $|P|=0$. Now, assume that $(i,j)$ is an element of $P$ and let $Q=P-\{(i,j)\}$. Then we can decompose $\PConf_{n,Q}(X)$ as the union of the open subscheme $\PConf_{n,P}(X)$ and its closed complement which is isomorphic to $\PConf_{n-1,P'}(X)$ for a certain $P'\in P(n-1)$. The result thus follows from the induction hypothesis and the purity theorem.
\end{proof}

\subsection{Betti realization and connectivity}

Finally, we will need a theorem which relates the connectivity of a mixed Tate motive to the connectivity of its Betti realization. 

\begin{theo}\label{theo: Betti detects connectivity}
Let $f \colon M\to N$ be a map in $\DAT(S,\Z)_{gm}$. Assume that $B^*(f)$ is an isomorphism in negative homological degrees. Assume further that one of the following conditions is satisfied
\begin{enumerate}
\item The ring $\Lambda$ is a $\Z/n$-algebra with $n$ any integer.
\item The ring $\Lambda$ is a subring of $\Q$.
\item The number $2$ is invertible in $\Lambda$.
\item The ring $\Lambda$ is arbitrary and the base number field $K$ cannot be embedded in $\R$.
\end{enumerate} 
Then, the map
\[\Hom_{\DA(S,\Lambda)}(N,\Lambda(q)[p])\to \Hom_{\DA(S,\Lambda)}(M,\Lambda(q)[p])\]
induced by $f$ is an isomorphism for all $q$ and for all $p<0$.
\end{theo}

\begin{proof}
We treat the case when $\Lambda$ is a $\Z/n$-algebra. Then Suslin's rigidity theorem (see \cite[Theorem 4.1]{ayoubrealisation}) gives us an equivalence of categories
\[\DA(S,\Lambda)\simeq \D(S_{et},\Lambda).\]
Moreover, viewed through the equivalence, the Betti realization functor can be identified with
\[i^*:\D(S_{et},\Lambda)\to\D(T_{et},\Lambda)\simeq\D(\Lambda)\]
where $T_{et}=\mathrm{Spec}(\overline{K})$ is the small étale site of the algebraic closure of $K$ and $i^*$ is the map induced by the inclusion $i \colon K\to\overline{K}$ (see \cite[Lemme 1.23 and Proposition 1.26]{ayoubalgebreII}). Moreover, for $U\in\D(S_{et},\Lambda)$, we have an étale descent spectral sequence of the form
\[H^s(\Gamma,\Hom_{\D(T_{et},\Lambda)}(i^*U,\Lambda(q)[t]))\implies \Hom_{\D(S_{et},\Lambda)}(U,\Lambda(q)[s+t])\]
where $H^*(\Gamma,-)$ denotes Galois cohomology with respect to $\Gamma=\mathrm{Gal}(\overline{K}/K)$ (see for instance \cite[VIII, Corollaire 2.3]{sga4}). Coming back to our situation, the map $f \colon M\to N$ will induce an isomorphism on the $E^2$-page of the étale descent spectral sequence in the range $t<0$. Since, moreover, this spectral sequence is zero for negative $s$, we see that the map $f \colon M\to N$ will induce an isomorphism on the $E^\infty$-page for $s+t<0$ which implies that we indeed have an isomorphism
\[\Hom_{\DA(S,\Z)}(N,\Z/n(q)[p])\to \Hom_{\DA(S,\Z)}(M,\Z/n(q)[p])\]
for $p<0$.

The case $\Lambda=\Q$ is classical and follows from the existence of the motivic $t$-structure on $\DAT(S,\Q)_{gm}$ \cite{levinetate} and the argument of \cite[Corollary 4.6]{horelmotivic} (beware that, contrary to what is stated in \cite{horelmotivic}, this corollary is incorrect with integral coefficients).

Now we treat the case of $\Lambda$ a subring of $\Q$. Using the extension of scalars adjunction
\[\DA(S,\mathbb{Z})\leftrightarrows\DA(S,\Lambda)\]
we see that it suffices to prove that the map induced by $f$
\[\Hom_{\DA(S,\Z)}(N,\Lambda(q)[p])\to \Hom_{\DA(S,\Z)}(M,\Lambda(q)[p])\]
is an isomorphism for all $q$ and for all $p<0$. 

Since $\Lambda$ is a subring of $\Q$ it is of the form $\Z[U^{-1}]$ for $U$ a set of primes. Let us denote by $V$ the complement of $U$ so that we have $\Q=\Lambda[V^{-1}]$. Let $I$ be the poset of integers whose prime factors are in $V$. Then we have a short exact sequence of $I$-diagrams of abelian groups
\[0\to\{m\mapsto \Lambda\}\to\{m\mapsto m^{-1}\Lambda\}\to \{m\mapsto \Lambda/m\}\to 0\]
where we denote by $m^{-1}\Lambda$ the $\Lambda$-submodule of $\Q$ generated by $m^{-1}$ and we have used the isomorphism $\Lambda/m\cong m^{-1}\Lambda/\Lambda$. Using this short exact sequence together with exactness of $I$-indexed colimits, we deduce a long exact sequence of abelian groups for any geometric motive $P$
\[\ldots\to\on{colim}_I\Hom(P,\Lambda/m(q)[p-1])\to \Hom(P,\Lambda(q)[p])\to \on{colim}_I\Hom(P,m^{-1}\Lambda(q)[p])\to\ldots\]
where all Homs are in $\DA(S,\Z)$. Now, we have a sequence of isomorphisms
\begin{align*}
\on{colim}_I\Hom(P,m^{-1}\Lambda(q)[p])&\cong \on{colim}_I(\Hom(P,m^{-1}\mathbb{Z}(q)[p]))\otimes\Lambda\\
&\cong (\Hom(P,\mathbb{Z}(q)[p]))\otimes \mathbb{Z}[V^{-1}]\otimes\Lambda\\
&\cong (\Hom(P,\mathbb{Z}(q)[p]))\otimes \Q\\
&\cong\Hom(P,\Q(q)[p])
\end{align*}
These isomorphisms follow from \cite[Corollary 5.4.9]{cisinskietale} (note that the authors work in the category $\DM_h(S,\Lambda)$ which is equivalent to $\DA(S,\Lambda)$ by \cite[Corollary 5.5.7]{cisinskietale}, note also that they denote by $\D_{\A^1,et}(S,\Lambda)$ our $\DA(S,\Lambda)$).
So our long exact sequence reduces to 
\[\ldots\to\on{colim}_I\Hom(P,\Lambda/m(q)[p-1])\to \Hom(P,\Lambda(q)[p])\to \Hom(P,\Q(q)[p])\to\ldots\]
The desired result will follow from this long exact sequence using the five lemma and the cases $\Lambda=\Q$ and $\Lambda=\Z/n$ that have already been treated.

Now, we treat the case of $\Lambda$ a ring in which $2$ is invertible. The short exact sequence
\[0\to \Lambda\to\Lambda\otimes\Q\to(\Lambda\otimes\Q)/\Lambda\to 0\]
induces a long exact sequence
\[\ldots\to\Hom(P,(\Lambda\otimes\Q)/\Lambda(q)[p-1])\to \Hom(P,\Lambda(q)[p])\to \Hom(P,\Lambda\otimes\Q(q)[p])\to\ldots\]
where all Homs are in $\DA(S,\Z[1/2])$. Now, we can use the fact that geometric motives are compact in $\DA(S,\Z[1/2])$ (this is proved in \cite[Proposition 8.3 and Remarque 3.13]{ayoubrealisation}). It follows that $\Hom(P,\Lambda\otimes\Q(q)[p])$ is a direct sum of copies of $\Hom(P,\Q(q)[p])$ and similarly that $\Hom(P,(\Lambda\otimes\Q)/\Lambda(q)[p-1]) $ is a filtered colimit of groups of the form $\Hom(P,A(q)[p-1])$
with $A$ a finite abelian group. So the result follows from cases that have already been treated.

Finally if $K$ does not embed in $\R$, then geometric motives are compact in $\DA(S,\Z)$ by \cite[Proposition 8.3 and Remarque 3.13]{ayoubrealisation}) and the proof in the paragraph above applies mutatis mutandis.
\end{proof}

\section{Construction of the stabilization map}
\label{stabilization}

In this section $X$ is a non-empty smooth scheme over $S$.

\begin{assu}\label{assumption}
We assume that there exists a pair $(Y,D)$ consisting of $Y$ a smooth geometrically connected scheme over $S$, $D$ a non-empty closed smooth subscheme such that $X\cong Y-D$. Finally, we assume that $D$ has a $K$-point. 
\end{assu}

The first part of the assumption is not very restrictive. Such a pair can be found as soon as $X$ is not complete. Indeed, in that case, by a theorem of Nagata, $X$ can be written as the complement of a non-empty closed subscheme in a complete scheme $\overline{X}$. Then, using Hironaka's resolution of singularities, we can assume that $\overline{X}$ is smooth and that $\Delta=\overline{X}-X$ is a normal crossing divisor. If we write $\Delta=\cup_{i=0}^nD_i$ the decomposition of $\Delta$ into irreducible components, we can then take $Y=\overline{X}-\cup_{i=1}^nD_i$ and $D=D_0\cap Y$.

Under Assumption \ref{assumption}, we construct (see Construction \ref{construction-stabilisation} below) a stabilization map
\begin{equation}\label{eq:stabilisation}
M(\PConf_n(X))\to M(\PConf_{n+1}(X))
\end{equation}
that is $\Sigma_n$-equivariant (see Lemma \ref{lemma-equivariance}) in the category $\DM(S,\Lambda)$. We can thus take homotopy orbits with respect to the symmetric group and get a map
\[M(\widetilde{\Conf}_n(X))\to M([\PConf_{n+1}(X)/\Sigma_n])\]
Finally we can compose this with the obvious map
\[M([\PConf_{n+1}(X)/\Sigma_n])\to M([\PConf_{n+1}(X)/\Sigma_{n+1}])\cong M(\widetilde{\Conf}_{n+1}(X))\]
and we get the stabilization map
\begin{equation}\label{eq:stabilisation-unordered}
M(\widetilde{\Conf}_n(X))\to M(\widetilde{\Conf}_{n+1}(X)).
\end{equation}
This is the map that will induce the isomorphisms in Theorems \ref{theorem-open-schemes} and \ref{main-theorem}. The maps \eqref{eq:stabilisation} and \eqref{eq:stabilisation-unordered} have the correct Betti realizations by Lemma \ref{lemma-Betti-realisations}. Observe that in the étale setting, and under the assumptions of Proposition \ref{prop : quotient vs stacky quotient}, the stabilization map above can equally  be seen as a map of étale motives:
\[M_{et}(\Conf_n(X))\to M_{et}(\Conf_{n+1}(X))\]

\begin{prop}
\label{prop:tubular}
Let $X$ be a smooth scheme and $D$ be a smooth closed subscheme of codimension $c$, let $N(j)$ be the normal bundle of the inclusion $j:D\to X$ and let $N_0(j)$ be the complement of the zero section of $N(j)$. Then there exists a map called the motivic exponential map
\[\mathrm{exp}^0:M(N_0(j))\to M(X-D)\]
whose Betti realization is homotopic to the map induced by the tubular neighborhood inclusion
\[C_*(N_0(j)_{an})\subset C_*((X-D)_{an}).\]
\end{prop}

\begin{proof}
In \cite[Section 5.2]{levinemotivic} Levine shows that one can construct an exponential map 
\[\mathrm{exp}^0:M(N_0(j))\to M(X-D)\]
such that the following square in $\DM(S,\Lambda)$ commutes and is homotopy cocartesian:
\begin{equation}
\label{eq:map-of-triangles-1}
\centering
\begin{split}
\begin{tikzpicture}
[x=1mm,y=1mm]
\node (tl) at (0,10) {$M(N_0(j))$};
\node (tm) at (30,10) {$M(N(j))$};
\node (bl) at (0,0) {$M(X-D)$};
\node (bm) at (30,0) {$M(X)$};
\draw[->](tl) to (tm);
\draw[->] (bl) to (bm);
\draw[->] (tl) to (bl);
\draw[->] (tm) to (bm);
\end{tikzpicture}
\end{split}
\end{equation}
where the right-hand vertical map is induced by the composition $N(j)\to D\xrightarrow{j} X$. Note that Levine constructs $\mathrm{exp}^0$ in the stable motivic homotopy category $\cat{SH}(S)$ instead of $\DM(S,\Lambda)$ but it is clear that his construction can be adapted to work in $\DM(S,\Lambda)$. Alternatively  $\DM(S,\Lambda)$ is simply the category of modules over the motivic Eilenberg-MacLane spectrum in $\cat{SH}(S)$ (this is the main theorem of \cite{rondigsmodules}) so we can simply tensor Levine's map with the motivic Eilenberg-MacLane spectrum to get the desired map.

The map $\mathrm{exp}^0$ has the correct Betti realization. Indeed, since the Betti realization functor can be modelled by a left adjoint $\infty$-functor between stable $\infty$-categories, it preserves homotopy cartesian squares. Therefore, applying Betti realization to the square \ref{eq:map-of-triangles-1}, we get a homotopy cartesian square in $\D(\Lambda)$:
\begin{equation}
\label{eq:map-of-triangles-2}
\centering
\begin{split}
\begin{tikzpicture}
[x=1mm,y=1mm]
\node (tl) at (0,10) {$C_*(N_0(j)_{an})$};
\node (tm) at (30,10) {$C_*(N(j)_{an})$};
\node (bl) at (0,0) {$C_*((X-D)_{an})$};
\node (bm) at (30,0) {$C_*(X_{an})$};
\draw[->] (tl) to (tm);
\draw[->] (bl) to (bm);
\draw[->] (tl) to (bl);
\draw[->] (tm) to (bm);
\end{tikzpicture}
\end{split}
\end{equation}
On the other hand, we know from the excision theorem in classical algebraic topology and the tubular neighborhood theorem that there is a homotopy cocartesian (and hence also cartesian) square in $\D(\Lambda)$ which is the same as the one above but where the left-hand vertical map is induced by the tubular neighborhood inclusion. Since homotopy pullbacks are unique up to weak equivalences, the Betti realization of $\mathrm{exp}^0$ must be homotopic to the map induced by the tubular neighborhood inclusion.
\end{proof}

\begin{prop}
\label{prop-equivariance}
In the setting of Proposition \ref{prop:tubular}, suppose that $X$ is equipped with an action of a discrete group $G$ that sends $D$ to itself. Equip the normal bundle $N(j)$ with the natural induced $G$-action, which restricts to a $G$-action on the subscheme $N_0(j)$. Then the motivic exponential map of Proposition \ref{prop:tubular} is $G$-equivariant.
\end{prop}

\begin{proof}
An examination of the construction of Levine (see \cite[Section 5.2]{levinemotivic}) shows that the map $\mathrm{exp}^0$ is natural in the data $(X,D)$.
\end{proof}

\begin{defi}
Denote by $\PConf_{n+1,\leqslant 1}(Y,D) \subset \PConf_{n+1}(Y)$ the smooth subscheme of ordered $(n+1)$-point configurations in $Y$ where the first $n$ points lie in $X$, and denote by $\PConf_{n+1,1}(Y,D)$ its closed smooth subscheme given by those configurations where, in addition, the last point lies in $D$. Note that $\PConf_{n+1,1}(Y,D)$ is isomorphic to $\PConf_n(X) \times D$ and its complement $\PConf_{n+1,\leqslant 1}(Y,D) - \PConf_{n+1,1}(Y,D)$ is isomorphic to $\PConf_{n+1}(X)$.
\end{defi}

\begin{cons}
\label{construction-stabilisation}
The stabilization map \eqref{eq:stabilisation} is constructed as follows. Denote the inclusion $D \to Y$ by $i$ and the inclusion $\PConf_{n+1,1}(Y,D) \to \PConf_{n+1,\leqslant 1}(Y,D)$ by $j$. Choose a $K$-point $* \in N_0(i)$. Such a point exists by our assumption that $D$ has a $K$-point and the fact that $D$ has positive codimension in $Y$ (indeed, $Y_{an}$ is connected and $Y_{an}-D_{an}$ non-empty, so $D_{an}$ must have positive codimension in $Y_{an}$ and dimension is preserved by analytification by \cite[Expos\'e XII, Proposition 2.1]{SGA1}). By construction, we have an identification $N_0(j) = \PConf_n(X) \times N_0(i)$, so the choice of $*$ induces a map
\begin{equation}
\label{eq:stabilisation-step1}
M(\PConf_n(X)) \to M(N_0(j)).
\end{equation}
By Proposition \ref{prop:tubular} and the identification of $\PConf_{n+1,\leqslant 1}(Y,D) - \PConf_{n+1,1}(Y,D)$ with $\PConf_{n+1}(X)$, there is a map
\begin{equation}
\label{eq:stabilisation-step2}
M(N_0(j)) \to M(\PConf_{n+1}(X)),
\end{equation}
and \eqref{eq:stabilisation} is defined to be the composition of these two maps.
\end{cons}

\begin{lemm}
\label{lemma-equivariance}
The stabilization map \eqref{eq:stabilisation} is $\Sigma_n$-equivariant.
\end{lemm}
\begin{proof}
There is an action of $\Sigma_n$ on $\PConf_{n+1,\leqslant 1}(Y,D)$ given by permuting the first $n$ points, and this action preserves its subscheme $\PConf_{n+1,1}(Y,D)$. This induces an action on $N_0(j)$, which corresponds, under the identification $N_0(j) = \PConf_n(X) \times N_0(i)$, to the permutation action on $\PConf_n(X)$ and the trivial action on $N_0(i)$. Hence the map \eqref{eq:stabilisation-step1} is $\Sigma_n$-equivariant. By Proposition \ref{prop-equivariance}, the map \eqref{eq:stabilisation-step2} is also $\Sigma_n$-equivariant.
\end{proof}

We refer the reader to \cite[\S 4]{Randal-Williams2013} for a construction of the stabilization map in a topological setting; see also \cite[p.~101]{mcduffconfiguration} for a more classical reference. We call this map the \emph{classical stabilization map} in the next lemma.

\begin{lemm}
\label{lemma-Betti-realisations}
The Betti realizations of \eqref{eq:stabilisation} and of \eqref{eq:stabilisation-unordered} are the maps of chain complexes induced by the classical stabilization maps for ordered and unordered configuration spaces respectively.
\end{lemm}
\begin{proof}
The classical stabilization map for ordered configuration spaces is homotopic to the composition of
\begin{equation}
\label{eq:stabilisation-classical-step1}
x \mapsto (x,*) \colon \PConf_n(X)_{an} \to \PConf_n(X)_{an} \times N_0(i)_{an} \cong N_0(j)_{an}
\end{equation}
with the tubular neighborhood inclusion
\begin{equation}
\label{eq:stabilisation-classical-step2}
N_0(j)_{an} \to \PConf_{n+1}(X)_{an}.
\end{equation}
(In more detail, the classical stabilisation map is defined via an auxiliary manifold with non-empty boundary whose interior is the manifold $X_{an}$. The composition of \eqref{eq:stabilisation-classical-step1} and \eqref{eq:stabilisation-classical-step2} is homotopic to this construction, where the auxiliary manifold is the (real) blow-up of $Y_{an}$ along $D_{an}$. Here we are using the fact that $D_{an}$ is non-empty, since $D$ has a $K$-point.)

We therefore need to show that $C_*(\eqref{eq:stabilisation-classical-step1};\Lambda) \cong B^*(\eqref{eq:stabilisation-step1})$ and $C_*(\eqref{eq:stabilisation-classical-step2};\Lambda) \cong B^*(\eqref{eq:stabilisation-step2})$. The first of these identifications follows directly from Fact \ref{fact-Betti-realisation} and the second follows from the last part of Proposition \ref{prop:tubular}. This shows that the Betti realization of \eqref{eq:stabilisation} is the map induced by the classical stabilization map for ordered configuration spaces.

The classical stabilization map $s^{un} \colon \Conf_n(X)_{an} \to \Conf_{n+1}(X)_{an}$ for unordered configuration spaces is obtained from the classical stabilization map $s^{ord} \colon \PConf_n(X)_{an} \to \PConf_{n+1}(X)_{an}$ for ordered configuration spaces by taking quotients with respect to the natural symmetric group actions. Since these actions are free and properly discontinuous, we may equivalently take homotopy orbits instead of quotients. At the beginning of this section, we defined \eqref{eq:stabilisation-unordered} by taking homotopy orbits (in the category of motives) of symmetric group actions on the map \eqref{eq:stabilisation}. The fact that $C_*(s^{un};\Lambda) \cong B^*(\eqref{eq:stabilisation-unordered})$ thus follows from the fact that $C_*(s^{ord};\Lambda) \cong B^*(\eqref{eq:stabilisation})$ -- proved in the paragraph above -- and the fact that both $B^*$ and $C_*(-;\Lambda)$ preserve homotopy colimits, since they are both left adjoints.
\end{proof}

\section{Homological stability for open schemes}
\label{etale}

In this section, we still work under Assumption \ref{assumption}. We further assume that we have chosen an $S$-point of the punctured normal bundle of the inclusion $D\to Y$. Hence, we have a stabilization map
\[M(\widetilde{\Conf}_n(X))\to M(\widetilde{\Conf}_{n+1}(X))\]
as explained in the previous section. If $X$ is a quasi-projective variety, then the assumptions of Proposition \ref{prop : quotient vs stacky quotient} are satisfied, and this stabilization map induces a map in {\'e}tale motivic cohomology
\[\H^{p,q}_{et}(\Conf_{n+1}(X),\Lambda)\cong \H^{p,q}_{et}(\widetilde{\Conf}_{n+1}(X),\Lambda)\to \H^{p,q}_{et}(\widetilde{\Conf}_n(X),\Lambda)\cong \H^{p,q}_{et}(\Conf_{n}(X),\Lambda).\]

\begin{theo}
\label{theorem-open-schemes}
Assume that $X$ satisfies the conditions above and assume further that $M_{et}(X)$ is a mixed Tate motive. Assume that one of the four conditions of Theorem \ref{theo: Betti detects connectivity} is satisfied. Then the stabilization map 
\[\H^{p,q}_{et}(\widetilde{\Conf}_{n+1}(X),\Lambda)\to \H^{p,q}_{et}(\widetilde{\Conf}_{n}(X),\Lambda)\]
is an isomorphism for $p\leq n/2$. If furthermore $X$ is a quasi-projective variety, then the stabilization map
\[\H^{p,q}_{et}(\Conf_{n+1}(X),\Lambda)\to \H^{p,q}_{et}(\Conf_{n}(X),\Lambda)\]
is an isomorphism for $p\leq n/2$.
\end{theo}

\begin{proof}
This follows from the fact that the Betti realization detects connectivity (Theorem \ref{theo: Betti detects connectivity}) and the statement in the topological case (Theorem \ref{t:classical}). In order to apply Theorem \ref{t:classical}, we must verify that the manifold $X_{an}$ is connected and open (i.e.~non-compact). By our assumptions, we have $X_{an} \cong Y_{an} - D_{an}$, where $Y_{an}$ is connected and $D_{an}$ is a non-empty (since $D$ has a $K$-point) closed submanifold; thus $X_{an}$ is non-compact. Moreover, since $D_{an}$ has positive complex codimension, thus real codimension at least $2$, cutting it out of $Y_{an}$ cannot disconnect, so $X_{an}$ is also connected.
\end{proof}

\section{Stability for $X=\A^d$}
\label{motivic}

In the case where $X=\A^d$ we can prove stability in the category $\DM(S,\Lambda)$ instead of $\DA(S,\Lambda)$. Note that a similar result was claimed in \cite{horelmotivic} but the proof is incorrect. It was based on the erroneous claim that the Betti realization functor on $\DMT(S,\Lambda)$ detects connectivity. Here we replace the Betti realization by the associated graded for the weight filtration which indeed detects connectivity and is still sufficiently close to the Betti realization in this case.

\subsection{The weight filtration}

For $M\in \DMT(S,\Lambda)$, we denote by
\[\ldots\to w_{\geq n}(M)\to w_{\geq n-1}(M)\to\ldots\to M\]
the weight filtration of $M$. If we denote by $w_{\geq n}\DMT(S,\Lambda)$ the  localizing subcategory of $\DMT(S,\Lambda)$ generated by the objects $\Lambda(i)$ with $i\geq n$, then $w_{\geq n}$ is by definition the right adjoint to the inclusion
\[w_{\geq n}\DMT(S,\Lambda)\to \DMT(S,\Lambda).\]
The fact that this right adjoint exists follows from the adjoint functor theorem. Observe that we have the following computation
\[
w_{\geq n}(\Lambda(i)[j]) = \begin{cases} \Lambda(i)[j] & \mathrm{if}\; i\geq n,\\
0 & \mathrm{else}.
\end{cases}
\]
This follows from the fact that, when $p>i$, we have
\[\Hom_{\DMT(S,\Lambda)}(\Lambda(p),\Lambda(i)[j])=0.\]

Note that Levine gives a construction of this right adjoint in \cite{levinetate} but he restricts to geometric mixed Tate motives (Levine's proof is stated for motives with rational coefficients but applies to motives with integral coefficients as observed by Kahn in \cite{kahnweight}). Nevertheless we can extend Levine's functor to the whole category $\DMT(S,\Lambda)$ by imposing that it commutes with filtered colimits. The fact that these two constructions agree follows from the following lemma.

\begin{lemm}
\label{lem:weight-filtration-commutes-with-colimits}
The functor
\[w_{\geq n}:\DMT(S,\Lambda)\to \DMT(S,\Lambda)\]
preserves arbitrary homotopy colimits.
\end{lemm}

\begin{proof}
Since this functor is exact, it suffices to prove that it preserves arbitrary direct sums. The category  $w_{\geq n}\DMT(S,\Lambda)$ is compactly generated and we can pick the set of motives $\Lambda(i)$ with $i\geq n$ as a set of compact generators of this category. Then a standard argument shows that it is enough to check that the inclusion
\[w_{\geq n}\DMT(S,\Lambda)\to \DMT(S,\Lambda)\] 
sends the set of compact generators of $w_{\geq n}\DMT(S,\Lambda)$ to compact objects which is obvious.
\end{proof}

We denote by $w_n(M)$ the homotopy cofiber of the map
\[w_{\geq n+1}(M)\to w_{\geq n}(M)\]

\begin{prop}\label{prop : pure weight}
Let $M$ be an object of $\DMT(S,\Lambda)$. Then the motive $w_n(M)$ is of the form $A(n)$ for some object $A$ of $\D(\Lambda)$.
\end{prop}

\begin{prop}
\label{p:connectivity}
Let $M\in\DMT(S,\Lambda)$. Assume that the weight filtration of $M$ is bounded, i.e., that $w_{\geq n}(M)=0$ for any large enough $n$ and $w_{\geq n}(M)\cong M$ for any small enough $n$. Assume further that $B^*\circ w_n(M)$ is connective for all $n$. Then for $m<0$, we have
\[\Hom_{\DM(S,\Lambda)}(M,\Lambda(q)[m])=0\]
\end{prop}

\begin{proof}
We shall prove that for each $n$ and for each negative $m$, we have
\[\Hom_{\DM(S,\Lambda)}(w_{\geq n}(M),\Lambda(q)[m])=0\]
Since $w_{\geq n}(M)$ is eventually isomorphic to $M$ by assumption, this will give the desired result. The fact that $w_{\geq n}(M)=0$ for $n$ large enough implies that we can prove this by descending induction on $n$. So we assume that
\[\Hom_{\DM(S,\Lambda)}(w_{\geq n+1}(M),\Lambda(q)[m])=0\]
and consider the exact sequence
\[\Hom(w_n(M),\Lambda(q)[m])\to \Hom(w_{\geq n}(M),\Lambda(q)[m])\to \Hom(w_{\geq n+1}(M),\Lambda(q)[m]).\]
From this exact sequence, we see that it is enough to prove that $\Hom(w_n(M),\Lambda(q)[m])=0$ to conclude the argument. For this we can use the previous proposition and our assumption that $B^*\circ w_n(M)=0$. These two facts together imply that $w_n(M)$ is of the form $A(n)$ for $A$ a connective object of $\D(\Lambda)$. Hence we are reduced to proving that for $A$ connective and $m<0$, we have
\[\Hom_{\DM(S,\Lambda)}(A(0),\Lambda(q-n)[m])=0.\]
This is true for $A=\Lambda$ by the Beilinson-Soul\'e vanishing conjecture (recall that the Beilinson-Soul\'e vanishing conjecture is a conjecture for general fields but a theorem for number fields thanks to Borel's computation of the algebraic $K$-theory of number fields with rational coefficients). This can be extended to a general $A$ using the observation that the category of connective chain complexes is the smallest subcategory of $\D(\Lambda)$ that is stable under homotopy colimits and contains $\Lambda$.
\end{proof}

\begin{rem}
The appeal to the Beilinson-Soul\'e vanishing conjecture is the only reason that we have to restrict to a number field in this section. The rest of the argument works for an arbitrary subfield of $\C$.
\end{rem}

\begin{coro}
\label{c:connectivity}
Let $M\to N$ be a morphism in $\DMT(S,\Lambda)$ where $M$ and $N$ have bounded weight filtration. Suppose that $B^*(w_n(M)) \to B^*(w_n(N))$ induces isomorphisms on homology in degrees $\leq d$ for all $n$. Then $M\to N$ induces isomorphisms on motivic cohomology in all bidegrees $(p,q)$ with $p\leq d-1$ and surjections when $p=d$.
\end{coro}
\begin{proof}
Denote the cofiber of $M\to N$ by $C$, so that $B^*(w_n(C))$ is the cofiber of $B^*(w_n(M)) \to B^*(w_n(N))$. The long exact sequence on homology and the assumption imply that $B^*(w_n(C))$ has trivial homology in degrees $\leq d$, in other words it is connective after shifting degrees down by $d+1$. Proposition \ref{p:connectivity} therefore implies that we have $\Hom_{\DM(S,\Lambda)}(C,\Lambda(q)[p])=0$ for all $p\leq d$. The result then follows from the long exact sequence obtained by applying $\Hom_{\DM(S,\Lambda)}(-,\Lambda(q)[p])$ to $M\to N\to C$.
\end{proof}

\subsection{The weight filtration of $\PConf_n(\A^d)$}

We will need to compute the weight filtration of the scheme $\PConf_n(\A^d)$. For this purpose, we introduce a definition. We take $\alpha$ a positive rational number and we write $\alpha=p/q$ with $p$ and $q$ two positive coprime integers.

\begin{defi}
We say that an object $M\in\DM(S,\Lambda)$ is $\alpha$-pure if the following two conditions are satisfied.
\begin{itemize}
\item If $n$ is not a multiple of $q$, then the map
\[B^*(w_{\geq n+1}(M))\to B^*(w_{\geq n}(M))\]
is an isomorphism.
\item If $n$ is a multiple of $q$ then the map
\[B^*(w_{\geq n}(M))\to B^*(M)\]
exhibits $B^*(w_{\geq n}(M))$ as the $\alpha n$-connective cover of $B^*(M)$.
\end{itemize}
\end{defi}

In other words, an object is $\alpha$-pure if, up to rescaling by $\alpha$, the weight filtration induces the Postnikov filtration upon application of the Betti realization functor. Observe that if $q\neq 1$, then, $B^*(M)$ has homology concentrated in degrees that are multiples of $p$.

\begin{example}
Take $M=M(\mathbf{P}^n)$. Then it is a classical computation that
\[M=\Lambda(0)\oplus\Lambda(1)[2]\oplus\ldots\oplus\Lambda(n)[2n],\]
and it follows that $M$ is $2$-pure.

For an example where $\alpha$ is not an integer, take $M=M(\A^d-\{0\})$. Then one has
\[M=\Lambda(0)\oplus\Lambda(d)[2d-1]\]
and we easily see that $M$ is $(\frac{2d-1}{d})$-pure.
\end{example}

\begin{prop}
\label{p:pure}
Write $\alpha=p/q$ with $p$ and $q$ two coprime positive integers.
\begin{enumerate}
\item If $M$ is $\alpha$-pure, then $B^*(w_{qm}(M))\cong \H_{ p m}(B^*(M))[p m]$ and $B^*(w_{n}(M))=0 $ if $n$ is not a multiple of $q$.
\item If $M$ is $\alpha$-pure, then $M(q)[p]$ is also $\alpha$-pure.
\item If $M$ and $P$ are $\alpha$-pure and $N\in \DM(\Lambda)$ fits in a cofiber sequence
\[M\to N\to P\]
then $N$ is also $\alpha$-pure.
\end{enumerate}
\end{prop}

\begin{proof}
We prove (1). If $n$ is not a multiple of $q$, then, by definition, the map 
\[B^*(w_{\geq n+1}(M))\to B^*(w_{\geq n}(M))\]
is an isomorphism, which implies that $B^*(w_{n}(M))=0$. Using Proposition \ref{prop : pure weight}, we deduce that $w_n(M)=0$. If $n=qm$, we apply $B^*$ to the cofiber sequence
\[w_{\geq qm+1}(M)\to w_{\geq qm}(M)\to w_{qm}(M).\]
By definition, $B^*(w_{\geq q m}(M))$ is the $pm$-connective cover of $B^*(M)$ and $B^*(w_{\geq q m+1}(M))$ is the $p(m+1)$-connective cover of $M$, it follows that $B^*(w_{qm}(M))\cong \H_{pm}(B^*(M))[pm]$.

The proof of (2) is elementary, once we observe that $w_{\geq n}(M(q))\cong w_{\geq n -q}(M)$.

In order to prove (3), we use an alternative characterization of $\alpha$-pure objects. An object $M$ of $\DM(\Lambda)$ is $\alpha$-pure if, for all $n$ in $\mathbb{Z}$, the homology of $B^*(w_{\geq n}(M))$ is concentrated in degrees $\geq \lceil \alpha n\rceil$ and the homology of $B^*(w_{<n}(M))$ is concentrated in degrees $<\lceil \alpha n\rceil$. Now, consider a cofiber sequence
\[M\to N\to P\]
with $M$ and $P$ $\alpha$-pure. This induces a cofiber sequence
\[B^*(w_{\geq n}(M))\to B^*(w_{\geq n}(N))\to B^*(w_{\geq n}(P))\]
with $B^*(w_{\geq n}(M))$ and $B^*(w_{\geq n}(P))$ concentrated in homological degrees $\geq \lceil \alpha n\rceil$. It follows that $B^*(w_{\geq n}(N))$ has homology concentrated in degrees $\geq \lceil \alpha n\rceil$. Similarly, $B^*(w_{< n}(N))$ has homology concentrated in degrees $<\lceil \alpha n\rceil$.
\end{proof}

\begin{lemm}
\label{lem:sPConf}
The object $B^*(w_k(M(\PConf_n(\A^d))))$ is trivial if $d$ does not divide $k$ and 
\[
B^*(w_{kd}(M(\PConf_n(\A^d)))) \cong \H_{k(2d-1)}(\PConf_n(\C^d))[k(2d-1)].
\]
\end{lemm}

\begin{proof}
By (1) of the previous proposition, it suffices to prove that the motive of $\PConf_n(\A^d)$ is $(\frac{2d-1}{d})$-pure. First, we recall a definition from \cite[Definition 7.5]{ciricietale}. We say that a finite collection $\{V_i\}_{i\in I}$ of codimension $d$ linear subspaces of an affine space $\mathbf{A}^r$ is a good arrangement of codimension $d$ subspaces if for each $J\subset I$, the codimension of the intersection $\bigcap_{j\in J}V_j$ is a multiple of $d$. Then one checks easily that the arrangement of diagonals in $(\A^d)^n$ is a good arrangement of codimension $d$ subspaces. We shall prove that the lemma holds in this more general situation.

So let us consider $X=\mathbf{A}^r-\bigcup_{i\in I}V_i$ the complement of a good arrangement of codimension $d$ subspaces. We proceed by induction on the cardinality of $I$. The result is obvious if $I$ is empty. Now assume that $I=J\sqcup \{i\}$. Let us denote by $Y=\mathbf{A}^r-\bigcup_{j\in J}V_j$. Then, $X$ is an open subset of $Y$ whose closed complement, denoted $Z$, is either empty or the complement of a good arrangement of codimension $d$ subspaces. We have a cofiber sequence
\[M(Z)(d)[2d-1]\to M(X)\to M(Y)\]
and the result follows from the induction hypothesis and (2) and (3) of Proposition \ref{p:pure}.
\end{proof}

\subsection{The main theorem}

\begin{theo}
\label{main-theorem}
The map
\[\H^{p,q}(\widetilde{\Conf}_{n+1}(\A^d),\Lambda)\to \H^{p,q}(\widetilde{\Conf}_{n}(\A^d),\Lambda)\]
is an isomorphism for $p\leq n/2 - 1$ and a surjection for $p\leq n/2$.
\end{theo}

\begin{proof}
By Corollary \ref{c:connectivity}, it suffices to prove that, for each $a$, the stabilization map
\[
B^*(w_a(M(\widetilde{\Conf}_{n}(\A^d)))) \to B^*(w_a(M(\widetilde{\Conf}_{n+1}(\A^d))))
\]
induces isomorphisms on homology in degrees $\leq n/2$. We observe that
\begin{align*}
B^*(w_a(M(\widetilde{\Conf}_{n}(\A^d))))
&= B^*(w_a(M([\PConf_{n}(\A^d)/\Sigma_n]))) && \text{by definition} \\
&\cong B^*(w_a(M(\PConf_{n}(\A^d))_{h\Sigma_n})) && \text{by Lemma \ref{lem:motive-of-stacky-quotient}} \\
&\cong B^*(w_a(M(\PConf_{n}(\A^d)))_{h\Sigma_n}) && \\
&\cong B^*(w_a(M(\PConf_{n}(\A^d))))_{h\Sigma_n} &&
\end{align*}
where the last two equivalences hold since $B^*$ and $w_a$ both commute with homotopy colimits; for $B^*$ this is because it is a left adjoint, for $w_a$ this follows from Lemma \ref{lem:weight-filtration-commutes-with-colimits}. This is the zero object whenever $d$ does not divide $a$, by Lemma \ref{lem:sPConf}, so we may assume that $a=kd$.

By Lemma \ref{lem:sPConf} and the spectral sequence $\H_*(G;\H_*(C)) \Rightarrow \H_*(C_{hG})$ for a $G$-equivariant chain complex $C$, we have
\[
\H_l(B^*(w_{kd}(M(\PConf_n(\A^d))))_{h\Sigma_n}) \cong \H_{l-k(2d-1)}(\Sigma_n , \H_{k(2d-1)}(\PConf_n(\C^d))),
\]
so it suffices to prove that the map
\[ \H_{l-k(2d-1)}(\Sigma_n , \H_{k(2d-1)}(\PConf_n(\C^d))) \to \H_{l-k(2d-1)}(\Sigma_{n+1} , \H_{k(2d-1)}(\PConf_{n+1}(\C^d))) \]
is an isomorphism in the range $l \leq n/2$. By Lemma \ref{lem:twisted-homstab} below, this map is an isomorphism in the range
\[
l \leq n/2 +k(2d-2).
\]
Note that $k(2d-2) \geq 0$, since $k\geq 0$ and $d\geq 1$, so we are done.
\end{proof}

\begin{lemm}
\label{lem:twisted-homstab}
For any dimension $m \geq 2$ and coefficient ring $\Lambda$, the stabilization map
\[ \H_q (\Sigma_n , \H_r(\PConf_n(\R^m);\Lambda)) \to \H_q (\Sigma_{n+1} , \H_r(\PConf_{n+1}(\R^m);\Lambda)) \]
is an isomorphism for $q \leq \frac{n}{2} - \frac{r}{m-1}$.
\end{lemm}
\begin{proof}
For fixed $m \geq 2$, the assignment $S \mapsto \PConf_S(\R^m)$, where $S$ is a finite set, naturally extends to a functor $\mathrm{FI}\sharp \to \htop$, by \cite[Proposition 6.4.2]{churchellenbergfarb}, and hence the assignment
\[ S \mapsto \H_r(\PConf_S(\R^m);\Lambda) \]
extends to a functor $\mathrm{FI}\sharp \to \rmod{\Lambda}$. We will show that this functor is \emph{polynomial of degree $\leq 2r/(m-1)$}. This will imply the stated result, by \cite[Theorem A]{palmer2018} with $M=\R^\infty$ and $X=*$, since $\mathcal{B}(\R^\infty,*) \simeq \mathrm{FI}\sharp$. (Theorem A of \cite{palmer2018} is stated only in the case $\Lambda = \Z$, so that $\rmod{\Lambda}$ is the category of abelian groups, but the results of \cite{palmer2018} generalise immediately to any abelian category, including $\rmod{\Lambda}$.)

We now show that the functor $T = \H_r(\PConf_{\bullet}(\R^m);\Lambda) \colon \mathrm{FI}\sharp \to \rmod{\Lambda}$ is polynomial of degree at most $d = \lfloor 2r/(m-1) \rfloor$. Recall that this means that $\Delta^{d+1} T = 0$, where $\Delta$ is the operation on functors from $\mathrm{FI}\sharp$ to an abelian category defined in \cite[\S 3.1]{palmer2018}. In the terminology of \cite[\S 1]{djament2016} this is equivalent to saying that $T$ is \emph{strongly polynomial of strong degree at most $d$}, when considered as a functor on the subcategory $\Theta = \mathrm{FI} \subset \mathrm{FI}\sharp$. (Note that the operation $\Delta$ of \cite[\S 3.1]{palmer2018} corresponds to the operation $\delta_1$ of \cite[\S 1]{djament2016} and, in the case $\mathcal{M} = \Theta$ of \cite[\S 5.2]{djament2016}, it suffices to consider only $\delta_1$, rather than $\delta_a$ for all objects $a$ of $\mathcal{M}$, since $\Theta$ is generated as a monoidal category by the object $1$.) Thus \cite[Proposition 4.4]{djament2016} implies that it is equivalent to prove that $T|_{\mathrm{FI}}$ is \emph{generated in degrees at most $d$}. This, in turn, is equivalent, by \cite[Remark 2.3.8]{churchellenbergfarb}, to the condition that $H_0(T|_{\mathrm{FI}})_i = 0$ for all $i>d$, where $H_0$ is the left adjoint of the inclusion of $\mathrm{Fun}(\mathrm{FB},\rmod{\Lambda})$ into $\mathrm{Fun}(\mathrm{FI},\rmod{\Lambda})$, where $\mathrm{FB}$ is the category of finite sets and bijections. The proof of the implication $\text{(iv)} \Rightarrow \text{(iii)}$ in Theorem 4.1.7 of \cite{churchellenbergfarb} shows that this condition will hold as long as, for all $n\geq 0$, the $\Lambda$-module
\[ T(n) = \H_r(\PConf_n(\R^m);\Lambda) \]
is generated by at most $O(n^d)$ elements. It therefore remains to verify this last condition.

We first consider the case $\Lambda = \Z$. By \cite[Theorem V.4.1]{fadellhusseini}, the $\N$-graded ring
\[ \H^*(\PConf_n(\R^m);\Z) \]
is generated by $\binom{n}{2}$ elements, all in degree $m-1$, subject to certain relations (the cohomological Yang-Baxter relations). It follows that the abelian group $\H^r(\PConf_n(\R^m);\Z)$ is trivial unless $r = i(m-1)$, in which case it is generated by the set of (commutative) monomials of degree $i$ in $\binom{n}{2}$ variables. The number of such monomials is at most
\[ \textstyle\binom{n}{2}^i \sim O(n^{2i}) = O(n^d), \]
and so $\H^r(\PConf_n(\R^m);\Z)$ is generated by at most $O(n^d)$ elements. Since the cohomology groups $\H^*(\PConf_n(\R^m);\Z)$ are free in all degrees, by \cite[Theorem V.1.1]{fadellhusseini}, and the space $\PConf_n(\R^m)$ has the homotopy type of a finite CW-complex (see \cite[\S VI.8--10]{fadellhusseini}), the universal coefficient theorem implies that the homology groups $\H_*(\PConf_n(\R^m);\Z)$ are also free in all degrees and $\H_r(\PConf_n(\R^m);\Z)$ has the same rank as $\H^r(\PConf_n(\R^m);\Z)$, so it is also generated by at most $O(n^d)$ elements.

Finally, going back to the case of an arbitrary ring $\Lambda$, the universal coefficient theorem implies that there is an isomorphism of $\Lambda$-modules
\[
\H_r(\PConf_n(\R^m);\Lambda) \cong \H_r(\PConf_n(\R^m);\Z) \otimes_\Z \Lambda ,
\]
so $\H_r(\PConf_n(\R^m);\Lambda)$ is generated as a $\Lambda$-module by at most $O(n^d)$ elements.
\end{proof}

\begin{rem}
In the proof above, we could alternatively have applied \cite[Theorem A]{randalwahlhomological} with $\mathcal{C} = \mathrm{FI}$, $A = \varnothing$, $X=\{*\}$ (where we may take $k=2$) and $N=0$, instead of \cite[Theorem A]{palmer2018} with $M=\R^\infty$ and $X=*$. However, in this way we could only have deduced the statement of the lemma in the smaller range of degrees $q \leq \frac{n}{2} - \frac{2r}{m-1} - 1$. This in turn would suffice for the proof of the main theorem (Theorem \ref{main-theorem}), \emph{except} in dimension $d=1$.

Another alternative to \cite[Theorem A]{palmer2018} would be \cite[Theorem 4.3]{betley}, \emph{if} the functor $\H_r(\PConf_{\bullet}(\R^m);\Lambda) \colon \mathrm{FI}\sharp \to \rmod{\Lambda}$ could be extended further to the category of finite sets and all partially-defined functions (not just partially-defined injections). However, there does not seem to be a natural extension to a functor on this larger category.
\end{rem}

\bibliographystyle{alpha}
\bibliography{biblio}

\end{document}